\documentclass[article]{elsart}

\usepackage{amssymb}
\usepackage{amsmath}
\usepackage{color}

\usepackage{enumerate}
\usepackage{ucs}
\usepackage{pstricks}
\usepackage{pstricks-add}
 \usepackage{pst-all}
 
\newtheorem{proposition}{Proposition}
 \newtheorem{theorem}{Theorem}
 
 \newtheorem{corollary}[theorem]{Corollary}
 \newtheorem{lemma}{Lemma}
 \newenvironment{proof}[1][Proof]{\textbf{#1.} }{\ \rule{0.5em}{0.5em}\vskip 12pt}

\begin{document}

\begin{frontmatter}




\title{Vertex-monochromatic connectivity of strong digraphs}


\author{Diego Gonz\'alez-Moreno}
\address{Departamento de Matem\'aticas Aplicadas y Sistemas, Universidad Aut\'onoma Metropolitana - Cuajimalpa.}
\ead{dgonzalez@correo.cua.uam.mx}
\thanks{This research was supported by CONACyT-M\'exico, project CB-222104}

\author{Mucuy-kak Guevara}
\address{Facultad de Ciencias, Universidad Nacional Aut\'onoma de M\'exico }
\ead{mucuy-kak.guevara@ciencias.unam.mx }
\thanks{This research was supported by PAPIIT M\' exico, project IN115816}

\author{Juan Jos\'e Montellano-Ballesteros}
\address{Instituto de Matem\'aticas, Universidad Nacional Aut\'onoma de M\'exico}
\ead{juancho@math.unam.mx}
\thanks{This research was supported by PAPIIT M\' exico, project IN104915}

\begin{abstract}
A vertex coloring of a strong digraph $D$  is a \emph{strong vertex-monochromatic connection coloring (SVMC-coloring)} if for every pair $u, v$ of vertices in $D$ there exists an $(u,v)$-monochromatic path  having  all the internal vertices of the same color. 
Let $smc_v(D)$ denote the maximum number of colors used in an SVMC-coloring of a digraph $D$.
In this paper  we determine the value of $smc_v(D)$ for the line digraph of a digraph. We also we give conditions to find the exact value of $smc_v(T)$, where $T$ is  a tournament.
\end{abstract}


\end{frontmatter}

\section{Introduction} 	

Caro and Yuster  \cite{CY11} introduced  the concept of monochromatic connection  of an edge colored graph. An edge-coloring of a graph  $G$ is a \emph{monochromatic-connecting coloring} if  there exists a monochromatic path between any two vertices of $G$. The study of  monochromatic connecting colorings arises from the rainbow connecting coloring problem,
in which rainbow paths are considered (a path is said to be \emph{rainbow} if no two edges of them are colored the same).
The monochromatic connection problem has also been studied in oriented graphs \cite{GMG17}.  An arc-coloring of a digraph $D$  is a \emph{strongly monochromatic connecting coloring} (SMC-coloring, for short) if for every pair $u, v$ of vertices in $D$   there exists  a directed $(u,v)$-monochromatic path and a  directed
$(v,u)$-monochromatic  path.
The \emph{strong monochromatic connection number} of a strong digraph $D$,
denoted by $smc(D)$, is  the maximum number of colors used in an SMC-coloring of  $D$.
Concerning the strong monochromatic connection number of an oriented graph the following result was proved in \cite{GMG17}.
 
\begin{theorem}\label{main}
Let $D$ be a strongly connected oriented graph of  size $m$, and let ${\Omega}(D)$ be  the minimum size of a strongly connected  spanning subdigraph of $D$.
Then
$$
smc(D)=m-\Omega(D)+1.
$$
\end{theorem}
Cai, Li and Wu \cite{QLW17} defined the vertex-version of the monochromatic connecting coloring concept.  A path in a vertex colored graph  is  \emph{vertex-monochromatic} if its internal vertices are colored the same. A vertex-coloring of a graph is a \emph{vertex-monochromatic  connecting coloring} (VMC-coloring) if there is a vertex-monochromatic path joining any two vertices of the graph.  
This concept also can be extended to digraphs.  A directed path in a vertex colored digraph  is  \emph{vertex-monochromatic } if its internal vertices are colored the same. A vertex-coloring of a digraph is a \emph{strongly vertex-monochromatic connecting coloring} (SVMC-coloring)
if for every pair $u$ and  $v$ of vertices in $D$ there exists a directed  $(u,v)$-vertex-monochromatic path and a directed $(v,u)$-vertex-monochromatic path.
The \emph{monochromatic vertex-connecting number} of a strong digraph  $D$, denoted by  $smc_v(D)$, is the maximum number of colors that can be used in a strongly vertex-monochromatic  connecting coloring of $D$.

For an overview of the monochromatic and rainbow connection subjects we refer the reader to  \cite{CJMZ08,LS,LW18}. 
 
In this work we study the SVMC-colorings of strong digraphs. The paper is organized as follows. In section \ref{definitions} some basic definitions and notations are given. In section \ref{bounds}  lower and upper bounds for $smc_v(D)$ are presented. In section \ref{linedigraph} we focus on the family of line digraphs. Finally, in section \ref{tournament} we study the  strong vertex-monochromatic connection number of strongly connected tournaments.

\section{Definitions  and Notation}\label{definitions}

All the digraphs considered in this work are simple; that is, digraphs with no parallel arcs, nor loops are considered.
 If $(u,v)$ is an arc of $D$, then we use either $uv$ or $u\rightarrow v$ denote it. Two  vertices $u$ and $v$ of a digraph are   \emph{adjacent} if 
$u\rightarrow v$ or $v\rightarrow u$.
 All walks, paths and cycles are to be considered directed.  A digraph is  \emph{connected} if its underlying graph is connected.
{A digraph $D$ is \emph{unilateral} if, for every pair $u,v\in V(D)$, either $u$
is reachable from $v$, or $v$  is reachable from $u$ (or both)}.
 A \emph{$p$-cycle} is a cycle of length $p$. The minimum integer $p$ for which $D$ has a $p$-cycle is the \emph{girth of D} and it is denoted by $g(D)$. 
 A digraph $D$ is \emph{strongly connected} or \emph{strong}  if for every pair of vertices $u, v \subseteq V(D)$, the vertex $u$ is reachable from $v$ and the vertex $v$  is reachable from $u$.  {Given a {strong} digraph $D$, we use  $\Omega(D)$ to denote the minimum size of a strongly connected  spanning subdigraph of $D$}.
  Let $u,v$ be two vertices of $D$. We say that $u$ \emph{dominates} $v$, or $v$ \emph{is dominated by} $u$, if $v\in N^+(u)$.
   A set of vertices $S\subset V(D)$ is a \emph{dominating set} if each vertex $v\in V(D)\setminus S$ is dominated by at least one vertex in $S$.
  A set of vertices $S$ is  an \emph{absorbing set} if for each vertex $v\in V(D)\setminus S$ there exists a vertex $u\in S$ such that $v\in N^-(u)$. An orientation of a complete graph is a \emph{tournament}. A subdigraph $H$ is said to be \emph{absorbing subdigraph} (\emph{dominating subdigraph}) if the set $V(H)$ is an absorbing set (dominating set) of $D$.

Let $D=(V, A)$ be a strong digraph. The {\it subdigraph induced} by a set of vertices $S$ is  denoted by $D[S]$.
Given a positive integer $p$, let $[p]=\{1, 2, \dots, p\}$.  A vertex {\it $p$-coloring} of $D$ is a surjective function  $\Gamma: V\rightarrow [p]$.  For each ``color'' $i\in [p]$ the set of vertices $\Gamma^{-1}(i)$ will be called the {\it chromatic class} (of color $i$), and  if $|\Gamma^{-1}(i)| =1$, the color $i$ and the chromatic class $\Gamma^{-1}(i)$ will be called {\it singular}.  
A subdigraph $H$ of $D$ will be called \emph{monochromatic} if  $A(H)$ is contained in a chromatic class.
A $p$-coloring  $\Gamma$ of $D$ is an {\it optimal coloring} if $p=smc(D)$ and $\Gamma$ is an SMC-coloring of $D$. 
For general concepts we may refer the reader to \cite{BG,B}. 

\section{Bounds for $smc_v(D)$}\label{bounds}

In this section upper and lower bounds for the strong vertex-monochromatic connection number of a digraph $D$ are given. 
 
The next proposition is  the digraph version of the bounds obtained \cite{QLW17} for the monochromatic vertex-connection number of a graph.
\begin{proposition}\label{cotatrivial}
Let $D$ be a strong digraph of order $n$ and diameter $d$. Then
\begin{itemize}
\item[i)]  $smc_v(D)=n$ if and only if $d\le 2$.
\item[ii)]  If $d\ge 3$ then $smc_v(D)\le n-d+2$.
\end{itemize}
 \end{proposition}
\begin{proof}
\begin{itemize}
\item[\emph{i)}] Let $\Gamma: V(D)\rightarrow [n]$ be an SVMC-coloring of $D$.  Let $u,v$ be two vertices in $D$ such that $d(u,v)=d$.
Let $P$ be a $(u,v)$-vertex-monochromatic path. Since $\Gamma$ is an SVMC-coloring of $D$ and all the vertices of $D$ have a different color it follows that the length of $P$ is at most two. Therefore $d(u,v)\le 2$ and the result follows.
If $d\leq 2$,  the coloring that assigns to every vertex a different color is an SVMC-coloring of $D$. 
\item[\emph{ii)}] Let $u$ and $v$ be two vertices in $D$ such that $d(u,v)=d$. Let $P$ be a vertex-monochromatic path connecting $u$ and $v$. Observe that there are at least $d-2$ vertices in $P$ using the same color, therefore $smc_v(D)\le n-(d-2)$ and the result follows.
\end{itemize}
\end{proof}
The following example shows that the upper bound of item \emph{ii)} of the above theorem is tight. 
Let $D$ be the digraph with vertex set $V(D)=\{v_1,v_2,\dots ,v_n\}$ and  arc set $A(D)=\{v_iv_1,v_2v_i: i=3,4,\dots n\}\cup \{v_iv_2\}$.  Observe that $\Omega_v(D)=3$,
 $diam(D)=3$ and $smc_v(D)=n-diam(D)+2=n-1>n-\Omega_v(D)+1$.
\begin{figure}[h!]
\centering
\begin{pspicture}(3.1,4.25)
        \psset{unit=1.1, nodesep=3pt}         
       \cnode*(0,0.5){2pt}{u}
        \cnode*(2,0.5){2pt}{v}
        \cnode*(1,1){2pt}{w}
        \cnode*(1,1.75){2pt}{z}
        \cnode*(1,3.25){2pt}{y}
        \rput(-.295,0.5){\rnode{ua.}{$v_1$}}
        \rput(2.295,.5){\rnode{vb.}{$v_2$}}
        \rput(1,1.25){\rnode{vb.}{$v_3$}}      
                \rput(1,2){\rnode{vb.}{$v_4$}}      
        \rput(1,2.7){\rnode{vd.}{$\vdots$}}
                \rput(1,3.62){\rnode{ve.}{$v_n$}}
        \rput(1,5){\rnode{w.}{}}	
        \rput(1,5.5){\rnode{z.}{}}
        \rput[bl](-0.755,3.1){$D:$}
        \ncline[linewidth=1.25pt]{->}{u}{v}
        \ncline[linewidth=1.25pt]{->}{w}{u}
        \ncline[linewidth=1.25pt]{->}{z}{u}
        \ncline[linewidth=1.25pt]{->}{x}{u}
           \ncline[linewidth=1.25pt]{->}{y}{u}
           
   \ncline[linewidth=1.25pt]{->}{v}{x}
        \ncline[linewidth=1.25pt]{->}{v}{w}
        \ncline[linewidth=1.25pt]{->}{v}{z}
        \ncline[linewidth=1.25pt]{->}{v}{y}

 
        \ncline{->}{u1}{w1}
        \ncline{->}{w1}{v1}
        \ncline{->}{w1}{z1}

\end{pspicture}
\caption{$smc_v(D)=n-diam(D)+2.$}\label{fig error}
\end{figure}
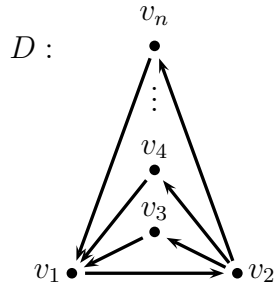 
\begin{theorem}
Let $D$ be a strong digraph and let  $\Gamma$ be an $SVMC$-coloring of $D$. Let $S$ be the set of singular chromatic classes of $\Gamma$ and let $D^*$ be the digraph induced by $V(D)\setminus S$.
\begin{itemize}
\item[$i)$] If for every vertex $v$ of $D$ there exists a vertex $x$ such that $d(v,x)\ge 3$, then $V(D^*)$ is a total absorbing set of $D$.
\item[$ii)$] If for every vertex $v$ of $D$ there exists a vertex $x$ such that $d(x,v)\ge 3$, then $V(D^*)$ is a total dominating set of $D$.
\item[$iii)$] If $g(D)\ge 5$, then $D^*$ is strong, absorbing and dominating set of $D$.
\end{itemize}
\end{theorem}
\begin{proof}
\emph{i)} Let $v$ be a vertex of $D$. Let $x\in V(D)$ such that $d(v,x)\ge 3$. Since $\Gamma$ is an SVMC-coloring there exists a $(v,x)$-vertex-monochromatic path $P=(v,x_1,\dots , x_r,x)$ with $r\ge 2$, such that the color of the internal vertices of $P$ belongs to a non-singular class, implying that $D^*$ is a total absorbing set.
\\
\emph{ii)} Let $v$ be a vertex of $D$ and let $x\in V(D)$ such that $d(x,v)\ge 3$. Let $P=(x,x_1,\dots , x_r,v)$, $r\ge 2$,  be a $(x,v)$-vertex-monochromatic path. Since the color of the internal verticres of $P$ are non-singular, then the set $V(D^*)$ is a total dominating set of $D$.
\\
\emph{iii)} Absorbing and dominating properties follows from items $i)$ and $ii)$. Let $u,v\in V(D^*)$ and suppose  that there is no $(u,v)$-path in $D^*$.
Since $\Gamma$ is an SVMC-coloring, there exists an  $(u,v)$-vertex-monochromatic path $P$ of length $2$ connecting and a $(v,u)$-vertex-monochromatic path $P'$ of length at least $3$ (because $g(D)\ge 5$). Suppose that $P'=(u,x_1,\dots ,x_{\ell},v)$. Note that the color used in the internal vertices of $P'$ is a non-singular color. Therefore $x_i\in V(D^*)$ for $i=1,2,\dots ,\ell$. Since $\Gamma$ is an SVMC-coloring of $D$ and $g(D)\ge 5$, it follows that the $(x_i,x_{i-1})$-vertex-monochromatic paths are totally contained in $D^*$. Hence $D^*$  is strongly connected.
\end{proof}

Let $D$ be a strong digraph and let $H$ be an absorbent, dominant and strong  subdigraph of $D$.  By coloring the vertices of $H$ with one single color and the remaining vertices with distinct colors, an SVMC-coloring of $D$ with $n-|V(H)|+1$ colors is obtained. 
Let $\Omega_v(D)$ denote the minimum order of an  absorbent, dominant and strong subdigraph of $D$. Therefore

\begin{equation}\label{cotaA}
smc_v(D)\ge n-\Omega_v(D)+1.
\end{equation}

\begin{theorem}\label{cota2}
Let $D$ be a strong digraph of order $n$ and girth $g(D)\ge 4$. Let $\Gamma$ be an SVMC-coloring of $D$ that uses $smc_v(D)$ colors. If $\ell$ is the minimum cardinality of a non-singular chromatic class of $\Gamma$, then 

$$
n-\Omega_v(D)+1\le smc_v(D)\le n-\Omega_v(D)+\frac{\Omega_v(D)}{\ell}.
$$
\end{theorem}
\begin{proof}
The left hand of the inequality is a consequence of (\ref{cotaA}).
Let $\Gamma$ be an SVMC-coloring of $D$ that uses $smc_v(D)$ colors. Let $a_i$ denote the number of the chromatic classes with cardinality $i$. Observe that $a_1+a_2+\dots +a_r=smc_{v}(D)$, where $r$ is the cardinality of the largest chromatic class . Let $\ell$ be the minimum cardinality of a non-singular chromatic class. Hence

$$
n=a_1+\sum_{i=\ell}^r ia_i\ge \ell (a_1+ \sum_{i=\ell}^r ia_i)-(\ell-1)a_i=\ell smc_v(D)-(\ell-1)a_1.
$$
Therefore,

$$
smc_v(D)\le \frac{n+(\ell-1)a_1}{\ell}.
$$

Let $D^*$ be the subdigraph of $D$ induced by the non-singular classes of $\Gamma$. By Theorem \ref{absdomcon} the set $V(D^*)$ is  strong, absorbing and dominating. Then 
$a_1+\Omega_v(D)\le n$. Hence 

$$
smc_v(D)\le \frac{n+(\ell-1)a_1}{\ell}\le \frac{n+(\ell-1)(n-\Omega_v(D))}{\ell}=n-\Omega_v(D)+\frac{\Omega_v(D)}{\ell}.
$$
\end{proof}

\begin{corollary}
Let $D$ be a strong digraph of order $n$. Then 

$$
n-\Omega_v(D)+1\le smc_v(D)\le n-\frac{\Omega_v(D)}{2}.
$$
\end{corollary}

\section{Line digraphs}\label{linedigraph}

Recall that the {line digraph} $L(D)$ of a  digraph $D = (V, A)$ has $A$ for its vertex set and $(e, f)$ is an arc in $L(D)$ whenever the arcs $e$ and $f$ in $D$ have a vertex in common which is the head of $e$ and the tail of $f$. 
A digraph $D$ is called a \emph{line digraph} if there exists a digraph $H$ such that $L(H)$ is isomorphic to $D$. 
In this section we determine the value of $smc_v(D)$ for a line digraph $D$.

\begin{proposition}\label{linea}
Let $D$ be a strong digraph and let $H$ be a spanning and strong subdigraph of $D$.  
If $L(H)$ is the subdigraph of  $L(D)$  induced by the arcs of $H$, then $L(H)$ is a strong, absorbing and dominating subdigraph of $L(D)$.
\end{proposition}
\begin{proof}
Let $D$ be a strong digraph and let $H$ be a spanning strong subdigraph of $D$. Since $H$ is strong it follows that $L(H)$ is a strong subdigraph of $L(D)$. 
Furthermore, since $H$ is an  spanning and strong subdigraph of $D$ for every  vertex $e=(u,v)$ of $L(D)$ there are two vertices $f_1$ and $f_2 $ of $L(H)$ such that $f_1=(w_1,u)$ and $f_2=(v,w_2)$. Therefore $f_1$ dominates the vertex $e$ and $f_2$ absorbs the vertex $e$.
\end{proof}

Let $H$ be the line digraph a of digraph $D$.
Let  $\Gamma: V(H)\longrightarrow [k]$ be an SVMC-coloring of $H$.  Notice that the coloring $\Gamma$ induces a coloring $\Gamma'$ of the arcs in $D$. Let $\Gamma':A(D)\longrightarrow [k]$ the coloring that assigns to each arc $e$ in $D$ the color $\Gamma (e)$ of the vertex $e\in V(H)$.

Let $D$ be a strong  digraph. An ordered pair $(u,v)$ of vertices of $D$ is said to be a \emph{bad pair} of $D$ if $N^+(u)=\{v\}$ and $N^-(v)=\{u\}$. Observe that if $(u,v)$ is a bad pair then ${uv}$ is an arc of $D$ and the pair $(v,u)$ is not a bad pair.

\begin{lemma}\label{goodpair}
Let $D$ be a strong digraph and let $H$ be the line digraph of $D$. 
Let $\Gamma$ be an SVMC-coloring of  $H$ and let $\Gamma'$ be the arc coloring of $D$ that assigns to each arc $e\in A(D)$ the color $\Gamma (e)$ of vertex $e\in V(H)$.
Given two vertices $u$ and $v$ in $D$ there exists an $(v,u)$-monochromatic path in $D$ if one of the following conditions holds.
\begin{itemize}
\item[{i)}]  The ordered pair $(u,v)$ is not a bad pair.
\item[{ii)}] The ordered pair $(u,v)$ is a bad pair and there exists an arc ${vw}$ in $D$ such that $(v,w)$ is not a bad pair.
\item[{iii)}] The ordered pair $(u,v)$ is a bad pair and there exists an arc ${wu}$ in $D$ such that $(w,u)$ is not a bad pair.
\item[{iv)}] If the previous cases do not happen and $D$ is different from $C_3$.
\end{itemize}
\end{lemma}
\begin{proof}
Let $D$ be a strong digraph and let $H=L(D)$. Let $u,v$ be two vertices of $D$. Let $\Gamma$ be an SVMC-coloring of $H$ and let $\Gamma'$ be the arc coloring of $D$ induced by $\Gamma$.
\begin{itemize}
\item[\emph{{i)}}] Suppose that $(u,v)$ is not a bad pair. Assume that there exists a vertex $w\in N^-(v)$ such that $w\neq u$. Since $D$ is strong there exists a vertex $w_1\in N^+(u)$ (it may happen that $w_1=v$). 
Since $\Gamma$ is an SVMC-coloring of $H$, there exists an vertex-monochromatic path $P=({wv},{vv_1},\dots , {v_{l-1}v_{l}},{v_lu}, {uw_1})$ connecting the vertices ${wv}$ and ${uw_1}$ of $H$. The path $P$ induces a monochromatic path connecting the vertices $v$ and $u$ in $D$.
If there exists a vertex $w\in N^+(u)$ such that $w\neq v$. 
Since $H$ is strong there exists a vertex $w_1\in N^-(v)$ and there is a vertex-monochromatic path $P$ connecting the vertices ${w_1v}$ and ${uw}$ which induces a monochromatic path connecting the vertices $v$ and $u$ in $D$.

\item[\emph{ii)}] Suppose that $(u,v)$ is a bad pair and there exists a vertex $w\in N^+(v)$ such that $(v,w)$ is not a bad pair. By the above item there is a $(w,v)$-monochromatic path $P$ of the same color of the arc ${uv}$. 
If $(w,u)$ is not a bad pair then there exists a $(u,w)$-monochromatic path $P'$ containing the arc $uv$ (because $(u,v)$ is a bad pair). The union of $P$ and $P'$ contains a $(v,u)$-monochromatic path. If $(w,u)$ is a bad pair, then $N^+(w)=\{u\}$ and $N^-(u)=\{w\}$. 
Since $(v,w)$ is not a bad pair there exists a vertex $z$ such that $z\in N^-(w)\setminus\{v\}$ or $z\in N^+(v)\setminus\{w\}$. Note that $(v,z)$,$(u,z)$  and $(z,u)$ are not  bad pairs because $N^+(v)\neq \{z\}$ and $v$ and $u$ and $z$ are not adjacent.
 Since $(v,z)$ and $(z,u)$ are not  bad pairs there is a $(z,v)$-monochromatic path $P$ and a $(u,z)$-monochromatic path $P'$ of the same color because both paths contains
 the arc ${uv}$. The union of $P$ and $P'$ contains a $(v,u)$-monochromatic path.

\item[\emph{iii)}] Suppose that the ordered pair $(u, v)$ is a bad pair and there exists an arc $wu$ in $D$ such
that $(w, u)$ is not a bad pair.  By item \emph{i)} there exists a $(u,w)$-monochromatic path $P$ that uses the arc $uv$ and therefore of the
same color of $uv$.  
If $(v,w)$ is not a bad pair, then by item \emph{i)} there exists a $(w,v)$-monochromatic $P'$ containing the arc $uv$ and therefore of the same color of $P$. 
The union of $P'$ and $P$ contains a $(v,u)$-monochromatic path in $D$. Continue assuming that $(v,w)$ is a bad pair. Therefore there is a vertex such that either $z\in N^+(w)$ or $z\in N^-(u)$. Observe that $(z,v),(v,z)$ and $(w,z)$ are not bad pairs. Hence there exists $(v,z)$-monochromatic path  a $(z,v)$-monochromatic path and a $(z,w)$-monochromatic path in $D$. The union of these paths contains a $(v,u)$-monochromatic path. 

\item[\emph{iv)}] If $D$ is isomorphic to $C_3$, then $H$  is also isomorphic to $C_3$. Since $smc_v(C_3)=3$ (see item \emph{i)} of Proposition \ref{cotatrivial}).  The coloring $\Gamma$ induces a coloring $\Gamma'$ of the arcs in $D$ with three colors that is not an SMC-coloring of $D$.
\end{itemize}
\end{proof}

\begin{theorem}
Let $D$ be a strong directed graph different from the cycle of length $3$. Then 
$$
scm_v(L(D))=smc(D).
$$
\end{theorem}
\begin{proof}
Let $D$ be a strong digraph of size $m$. By Theorem \ref{main} it follows that $smc(D)=m-\Omega(D)+1$ . Let $H$ be a strong and spanning subdigraph of $D$ of size  $\Omega(D)$.  By Proposition \ref{linea} it follows that  $L(H)$ is a strong, absorbing and dominanting subdigraph of $L(D)$. By  (\ref{cotaA}) it follows   
 $$
 smc_v(L(D))\ge |V(L(D))|-|V(H)|+1=m-\Omega(D)+1=smc(D).
 $$
 Observe that if $D$ is different from $C_3$, then every pair of vertices of $D$ satisfies one of the items of Lemma \ref{goodpair}. Hence, if $\Gamma$ is an SVMC-coloring of $L(D)$ it follows that the coloring  $\Gamma'$ of $D$ induced $\Gamma$ is an SMC-coloring of $D$ and therefore  $smc_v(L(D))\le smc(D)$, an the result follows.
\end{proof}

\section{Monochromatic vertex-connecting number of tournaments}\label{tournament}

In this section a condition on  $\Omega_v(T)$ of a strong tournament $T$ is given in order to find the exact value of $smc_v(T)$.
{
\begin{theorem}
Let $T$ be a strong tournament of diameter $d\ge 6$. If $\Omega(T)\le 2d-6$, then 
$$
smc_v(T)=n-\Omega_v(T)+1.
$$
\end{theorem}
\begin{proof}
Let $\Gamma$ be an SVMC-coloring of $T$ and let $u,v$ be two vertices of $T$ such that $d(u,v)=d\ge 6$. 
Let $P=(u,x_1,x_2,\dots , x_s,v)$ be a  $(u,v)$-vertex-monochromatic path. Since $P$ is a $(u,v)$-path of $T$, it follows that $s\ge d-1$. Observe that  the subdigraph induced by  $\{x_1,x_2,\dots , x_s\}$ is strong. 
Let $H$ be the biggest strong sudigraph of $T$ containing the set $\{x_1,x_2,\dots , x_s\}$ 
such that all the vertices of $H$ are colored the same. We claim that $V(H)$ is an absorbing and dominating set of $T$.
\\
\emph{Claim 1.} \emph{$V(H)$ is an absorbing set of $T$}. Suppose that there exists a vertex $w\in V(T)\setminus V(H)$ such that 
$x\rightarrow w$ for every vertex $x\in V(H)$. Since $\Gamma$ is an SVMC-coloring of $T$ there exists a $(w,v)$-vertex-monochromatic path $P'=(w,y_1,y_2,\dots ,y_r,v)$. Note that $(u,x_1,w,y_1,\dots , y_r,v)$ is a $(u,v)$-path, hence $r\ge d-3$. If the color of the internal vertices of $P'$ is different to the color of the vertices in $H$, then 
$$
smc_v(D)\le n-|V(H)|-(|V(P')|-2)+2\le n-(d-1)-(d-3)+2=n-2d+6.
$$
Combining the above inequality with  (\ref{cotaA}) it follows that 
$$
n-(2d-6)+1\le n-\Omega_v(T)+1\le smc_v(T)\le n-2d+6,
$$
giving a contradiction. 
Hence, the color of the internal vertices of $P$ is equal to the color of the vertices of $H$. Observe that $x_s\rightarrow y_1$, otherwise $(u,x_1,w,y_1,x_s,v)$ is a $(u,v)$-path of length $5$ contradicting that $d(u,v)=d\ge 6$. Furthermore, for  every $y_i\in V(P')$, $i=2,\dots s$, it follows that $x\rightarrow y_i$. If $y_i\rightarrow x$ for some $i=2,\dots ,s$, the subdigraph induced by $V(H)\cup \{y_1,y_2, \dots ,y_s\}$ would be a strong subdigraph of $T$ bigger than $H$, contradicting the election of $H$. 
Therefore $y\notin V(H)$ for every $y\in V(P')$ and
$$
smc_v(D)\le n-|V(H)|-(|V(P')|-2)+1\le n-(d-1)-(d-3)+1=n-2d+5,
$$
and using (\ref{cotaA}), a contradiction is obtained.

\emph{Claim 2.} 
\emph{$V(H)$ is a dominating set of $T$}. Suppose that there exists a vertex $w\in V(T)\setminus V(H)$ such that 
$w\rightarrow x$ for every $x\in V(H)$. Let $P'=(u,y_1,y_2,\dots, y_r,w)$ be a $(u,w)$-vertex-monochromatic path.
Since 
$(u,y_1,y_2,\dots $ $,w, x_s,v)$ is a $(u,v)$-path, it follows that  $s\ge d-3$.  If the color of the internal vertices of $P'$ is  different  to the color of the vertices of $H$, using a similar reasoning as in the proof of \emph{Claim 1} a contradiction is obtained.
Hence, the color of the internal vertices of $P$ 
is equal to the color of the vertices of $H$. Observe that $x_s\rightarrow y_1$, otherwise $(u,y_1,x_s,v)$ is a $(u,v)$-path of length 4, a  contradiction. Moreover, for every $y_i\in V(P')$, $i=2,\dots , y_s$, it follows that  $x\rightarrow y_i$, for every $x\in V(H)$. If not, the digraph induced by $V(H)\cup \{y_1,y_2, \dots ,y_i\}$ is a strong subdigraph of $T$ bigger than $H$, giving a contradiction. Therefore $y\notin V(H)$ for every $y\in V(P')$ and using a reasoning analogous to the proof of \emph{Claim 1} the result is followed.

Hence $H$ is an  absorbent, dominant and strong subdigraph of $T$. Since $\Gamma$ is an optimal SVMC-coloring of $T$ that assign the same color to every vertex in $H$, it follows that  $smc_v(T)=n-\Omega_v(T)+1$ and the result follows. 
\end{proof}}

\noindent {\bf Acknowledgments.} {}
 This research was supported by CONACyT-M\'exico, under project CB-222104.

\end{document}